\documentclass[12pt, a4paper]{amsart}

\usepackage[utf8]{inputenc}
\usepackage{amsfonts}
\usepackage{amssymb}
\usepackage{amsmath}
\usepackage{newlfont}
\usepackage[mathcal]{euscript}
\usepackage{tikz}
\usepackage{enumerate}
\usepackage{hyperref}
\usepackage{amsthm}
\usetikzlibrary{decorations.pathreplacing}
\usetikzlibrary{patterns}
\usepackage{pst-platon}

\usetikzlibrary{backgrounds}
\usetikzlibrary{arrows,shapes,positioning}
\usetikzlibrary{decorations.markings}
\tikzstyle arrowstyle=[scale=1]
\tikzstyle directed=[postaction={decorate,decoration={markings,
    mark=at position .65 with {\arrow[arrowstyle]{stealth}}}}]
\tikzstyle reverse directed=[postaction={decorate,decoration={markings,
    mark=at position .65 with {\arrowreversed[arrowstyle]{stealth};}}}]
    \usetikzlibrary{calc}

 \newtheorem{thm}{Theorem}

\newtheorem*{remark*}{Remark}

 \newcommand{\R}{\mathbb{R}}

\numberwithin{equation}{section}

\begin{document}

\title{Vertex to vertex geodesics on platonic solids}

\author{Serge Troubetzkoy}
\address[]{Aix Marseille Univ, CNRS, I2M, Marseille, France}
\address[]{Postal address: I2M, Luminy, Case 907, F-13288 Marseille Cedex 9, France}
\email{serge.troubetzkoy@univ-amu.fr}
\maketitle
\begin{abstract}
We give a simple proof  based on symmetries that there are no geodesics from a vertex to itself in the cube, tetrahedron, octahedron,  and icosahedron.
\end{abstract}

A straight-line trajectory on  the surface of a polyhedron is a straight line within a
face that is uniquely extended over an edge so that the trajectory forms a straight line in the plane
when the adjacent faces are unfolded to lie in the same plane.
This is well-defined away from the
vertices. Locally, a straight-line trajectory is the shortest curve between points, thus
it is a geodesic.  
By choosing a tangent vector at a vertex, one can consider the corresponding
geodesic emanating from that vertex. Thus, while geodesics can start and end at a vertex, they can not pass through a vertex.  The study of geodesics on polyhedra was initiated quite some time ago in \cite{S,R}

We give a short simple proof of the following fact first proved in \cite{DDTY} and \cite{F} and described in the expository article \cite{AA1}  as well as in the unpublished problem book \cite{P} where the question is attributed to  Jaros{\l}aw K\c{e}dra (starting with version 8 of the book (2016) for the cube and version 10 of the book for the other polyhedra (2018)).  A proof close to ours, but 
dressed up in advanced terminology, is given in \cite{AAH}.

\begin{thm}
There are no geodesics connecting a vertex to itself  on the cube,  tetrahedron, octahedron, or icosahedron.
\end{thm}

\begin{proof} The edges of all the polygons will be normalized to have length 1.
We begin with the cube.
Consider a geodesic segment $\gamma$ which starts at a vertex and ends at a vertex.  We  will show that the two vertices cannot coincide. For this we unfold the the geodesic, in our unfoldings the squares will be parallel to the coordinate axes
 and the geodesic will start at a vertex of a square placed at the origin.
 The unfolding of $\gamma$  is a line segment starting at the origin and ending  at vertex with coordinates $(p,q) \in \mathbb{N}^2$.
 The midpoint  $m$ of the geodesic segment has coordinates $(p/2,q/2)$.  Since geodesics do not pass through vertices either $p$ or $q$ must be odd.

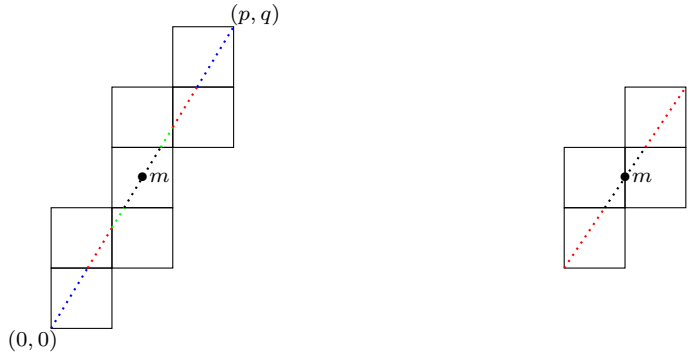
\begin{figure}[ht]
\begin{minipage}[ht]{.5\linewidth}
\centering
\begin{tikzpicture}[scale=0.8]
\node at (-1.3,-2.2) {\tiny $(0,0)$};
\draw
(-1,-2) rectangle (0,-1);
\draw 
(-1,-1)  rectangle (0,0);
\draw[thin] 
(0,-1) rectangle (1,0);
\draw (0,0) rectangle (1,1);
\node at (0.5,0.5) {\tiny $\bullet$};
\node at (0.78,0.5) {\tiny $m$};
\draw[thin] 
(0,1) rectangle (1,2);
\draw (1,1) rectangle (2,2);
\draw (1,2) rectangle (2,3);
\node at (2.35,3.2) {\tiny $(p,q)$};
\draw[thick, dotted, blue] (-1,-2) -- (-2/5,-1);
\draw[thick, dotted, red]  (-2/5,-1) -- (0,-1/3);
\draw[thick, dotted, green]   (0,-1/3) -- (1/5,0);
\draw[thick, dotted]  (1/5,0) -- (4/5,1);
\draw[thick, dotted, green]  (4/5,1) -- (1,4/3);
\draw[thick, dotted, red]  (1,4/3) -- (7/5,2);
\draw[thick, dotted, blue]   (7/5,2) -- (2,3);

\end{tikzpicture}
\end{minipage}\nolinebreak
\begin{minipage}[ht]{.5\linewidth}
\centering
\begin{tikzpicture}[scale=0.8]
\draw (0,0) rectangle (1,1);
\draw (0,1) rectangle (1,2);
\node at (1,1.5) {\tiny $\bullet$};
\node at (1.28,1.5) {\tiny $m$};
\draw (1,1) rectangle (2,2);
\draw (1,2) rectangle (2,3);
\draw[thick, dotted, red] (0,0) -- (2/3,1);
\draw[thick, dotted] (2/3,1) -- (1,3/2);
\draw[thick, dotted] (1,3/2) -- (4/3,2);
\draw[thick, dotted, red] (4/3,2) -- (2,3);
\end{tikzpicture}
\end{minipage}
\caption{Center of symmetry of the unfolding.}\label{fig1}
\end{figure}

 If both $p$ and $q$ are odd then $m$ is the center of one of the  squares of the unfolding (Figure \ref{fig1} left), and thus
 the midpoint $M$ of $\gamma$ is located in the center of one of the faces of the cube. 
If $p$ is even and $q$ is odd; then the midpoint of the unfolding is located in the middle of a vertical edge of one of the squares of the unfolding (Figure \ref{fig1} right),  while if $p$ is odd and $q$ is even then it is located in the middle of one of the horizontal edges of the unfolding.  In both of these last two cases the corresponding point $M$ on the cube is in the middle of one of the edges of the cube.  

In all  three cases the unlabeled unfolded figure is centrally symmetric about the point $m$.  We will show that refolding this leads to a symmetry of the geodesic $\gamma$. 

 Suppose first that $p$ and $q$ are odd, so $m$ is the center of a square.  We consider the unfolding embedded in $\R^3$, contained in the plane $z=0$. The two dimension central symmetry can be interpreted in $\mathbb{R}^3$ as the rotation by $180^{\circ}$ around the line
$L$ through $m$ perpendicular to face of the unfolding containing $m$.
Consider the point $m$ and follow the unfolded trajectory starting at $m$ in both directions; we arrive at the first pair of centrally symmetric edges and we refold them, 
the resulting object is again a geodesic which is invariant
 under a rotation by $180^{\circ}$ about the line $L$. 
We repeat this
 procedure each time we reach a pair of symmetric edges, in the end we obtain the
geodesic $\gamma$ on the cube, and since the symmetry is preserved at each step, we conclude that $\gamma$ is invariant under a rotation by $180^{\circ}$ about $L$.
Since $L$ passes through the center $m$ of a side and is perpendicular to this side, it  passes through the center of the cube and the center of the opposite face.
 We conclude that the two endpoints of $\gamma$ are a rotation of each other and thus cannot coincide (Figure \ref{fig2} top).

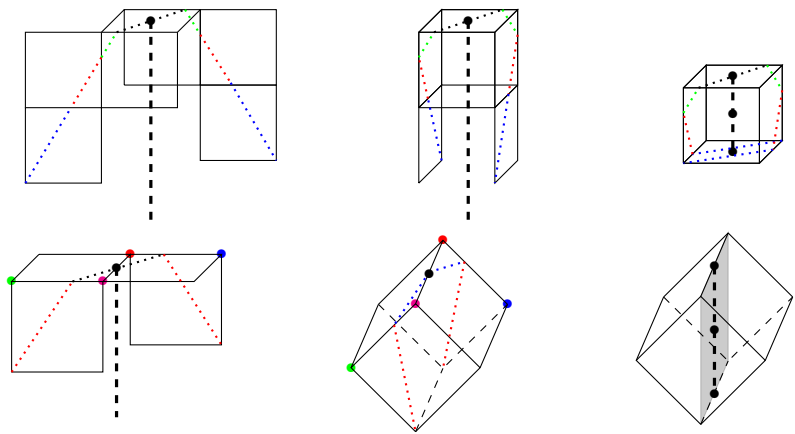
\begin{figure}[ht]
\begin{minipage}[ht]{.33\linewidth}
\centering
\begin{tikzpicture}[scale=1]

     \coordinate (C2) at (3.3, 0.3);
     \coordinate (C3) at (2.3,1.3);
     \coordinate (C4) at (3.3, 1.3);
      \coordinate (D1) at (1.3, -0.7);
      \coordinate (D2) at (2.3, -0.7);
     \coordinate (E1) at (-1,0);
       \coordinate (E2) at (-1,1);

 \draw (0,1) -- (1,1) -- (1.3,1.3);
 \draw (0,1) -- (-1,1) -- (-1,0);
 \draw (1,1) -- (1,0) -- (0,0) -- (0,1);
 \draw (0,1) -- (0.3,1.3);
 \node at (0.65,1.15)  { \tiny $\bullet$};
 \draw[very thick, dashed] (0.65,1.15) -- (0.65,-1.5);
 
    \draw (1.3,0.3) --  (2.3,0.3) --   (D2) -- (D1) -- (1.3,0.3);
            \draw (0,0) -- (0,-1) -- (-1,-1) -- (-1,0) -- (0,0);
\draw[] (0.3,0.3) -- (0.3,1.3) -- (1.3,1.3)  -- (1.3,0.3) -- (0.3,0.3);

\draw (1.3,1.3) -- (C3)  --  (2.3,0.3) -- (1.3,0.3) -- (1.3,1.3);
\draw[thick, dotted, black]   (0.2,1) -- (1.1,1.3) ;
\draw[thick, dotted, blue]  (-1,-1) --  (-2/5,0) ;
      \draw[thick, dotted, red]   
       (0,0.667)  -- (-2/5,0);
  \draw[thick, dotted, red ]   (1.3, 0.9667) -- (1.7,0.33) ;
\draw[thick, dotted, blue] (1.7,0.33) -- (D2);
\draw[thick, dotted, green]  (1.1,1.3) -- (1.3, 0.9667);
\draw[thick, dotted, green]    (0,0.667) -- (0.2,1) ;
        
\end{tikzpicture}
\end{minipage}\nolinebreak 
\begin{minipage}[ht]{.33\linewidth}
\centering
\begin{tikzpicture}[scale=1]

     \coordinate (C2) at (0.3, -0.7);
     \coordinate (C3) at (1,-1);
     \coordinate (C4) at (1.3, -0.7);
     
      \node at (0.65,1.15)  {  \tiny $\bullet$};
 \draw[very thick, dashed] (0.65,1.15) -- (0.65,-1.5);

\draw (0,1) -- (1,1) -- (1.3,1.3)  -- (0.3,1.3) -- (0,1);  
      
\draw[] (0,0) -- (0,1) -- (1,1)  -- (1,0) -- (0,0);
\draw[] (0.3,0.3) -- (0.3,1.3) -- (1.3,1.3)  -- (1.3,0.3) -- (0.3,0.3);
\draw (0,0) -- (0,1) -- (0.3,1.3) -- (0.3,0.3) -- (0,0);    
\draw (1,0) -- (1,1) -- (1.3,1.3) -- (1.3,0.3) -- (1,0);
\draw (0,-1) -- (C2) -- (0.3,0.3)  -- (0,0) -- (0,-1);
\draw (1.3,0.3) -- (C4)  -- (C3) -- (1,0) -- (1.3,0.3);

\draw[thick, dotted, black]   (0.2,1) -- (1.1,1.3) ;
\draw[thick, dotted, blue]  (C2) --  (0.12,0.12) ;
      \draw[thick, dotted, red]   (0.12, 0.12) --  (0,0.667) ;
  \draw[thick, dotted, red ]   (1.3, 0.9667) -- (1.18,0.18) ;

\draw[thick, dotted, blue] (1.18, 0.18) -- (C3);
\draw[thick, dotted, green]  (1.1,1.3) -- (1.3, 0.9667);
\draw[thick, dotted, green]    (0,0.667) -- (0.2,1) ;
        
\end{tikzpicture}
\end{minipage}\nolinebreak 
\begin{minipage}[ht]{.22\linewidth}
\centering
\begin{tikzpicture}[scale=1]
 
\begin{scope}
\draw (0,1) -- (1,1) -- (1.3,1.3)  -- (0.3,1.3) -- (0,1);  
\draw[thin] (0,0) -- (0,1) -- (1,1)  -- (1,0) -- (0,0);
\draw[thin] (0.3,0.3) -- (0.3,1.3) -- (1.3,1.3)  -- (1.3,0.3) -- (0.3,0.3);
\draw (0,0) -- (0,1) -- (0.3,1.3) -- (0.3,0.3) -- (0,0);    
\draw (1,0) -- (1,1) -- (1.3,1.3) -- (1.3,0.3) -- (1,0);
\draw[] (0,0) --(1,0) -- (1.3,0.3) -- (0.3,0.3) -- (0,0);
 \node at (0.65,1.15)  {  \tiny $\bullet$};
 \draw[very thick, dashed] (0.65,1.15) -- (0.65,0.15);
  \node at (0.65,0.15)  { \tiny $\bullet$};
    \node at (0.65,0.65)  {  \tiny $\bullet$};
    
\draw[thick, dotted, blue]  (1.18,0.18) -- (0,0);
\draw[thick, dotted, blue]  (1.3,0.3) -- (0.12, 0.12);
\draw[thick, dotted, red]   (0.12, 0.12) --  (0,0.667) ;
\draw[thick, dotted,red ]   (1.3, 0.9667) -- (1.18,0.18) ;
\draw[thick, dotted, green ]   (0,0.667) -- (0.2,1) ;
\draw[thick, dotted, green ]   (1.1,1.3) -- (1.3, 0.9667) ;
\draw[thick, dotted, black ]   (0.2,1) -- (1.1,1.3) ;
\end{scope}
\end{tikzpicture}
\end{minipage}
 \bigskip
\begin{minipage}[ht]{.33\linewidth}
\centering
\begin{tikzpicture}[scale=1.2]

  \node at (1,1) {\color{magenta} \tiny $\bullet$};
    \node at (1.3,1.3) {\color{red} \tiny $\bullet$};
     \node at (2.3,1.3) {\color{blue} \tiny $\bullet$};
        \node at (0,1) {\color{green} \tiny $\bullet$};
        
        \node at (1.15,1.15) {\tiny $\bullet$};
        \draw [very thick, dashed] (1.15,1.15) -- (1.15,-0.5);
        
\draw[] (0.3,1.3) -- (1.3,1.3) -- (2.3,1.3) -- (2,1) -- (1,1) -- (0,1) -- (0.3,1.3);
\draw[] (1,1) -- (1.3,1.3) --(1.3,0.3) -- (2.3,0.3) -- (2.3,1.3);
\draw[] (1,1) -- (1,0) -- (0,0) -- (0,1);
 \draw[thick, dotted]  (2/3,1) -- (5/3,1.3);
\draw[thick, dotted, red] (5/3,1.3) -- (2.3,0.3);
\draw[thick, dotted, red] (0,0) -- (2/3,1);

\end{tikzpicture}
\end{minipage}\nolinebreak
\hspace{-0.2cm}
\begin{minipage}{0.33\linewidth}
\centering
\begin{tikzpicture}[scale=1.2]

 \coordinate (v) at (0.3,0.7);
 \coordinate (w) at (0.15,0.33);
     \coordinate (A) at (0,0);
     \coordinate (B) at (0.707,0.707);
     \coordinate (C) at (0,1.414);
     \coordinate (D) at (-0.707,0.707);     
     \coordinate (E) at ($(A) + (v)$);
          \coordinate (F) at ($(B) + (v)$);
               \coordinate (G) at ($(C) + (v)$);
                    \coordinate (H) at ($(D) + (v)$);
                    \coordinate (I) at ($(C) + (w)$);
                        \coordinate (J) at ($(A) + (w)$);
            \node at (C) {\color{magenta} \tiny $\bullet$};
             \node at (G) {\color{red}  \tiny $\bullet$};
              \node at (F) {\color{blue} \tiny $\bullet$};
      \node at (D) {\color{green} \tiny $\bullet$};

\draw (A) -- (B) -- (C) -- (D) -- (A);
\draw[dashed]  (H) -- (E) --(F);
\draw (F) -- (G) -- (H);
\draw[dashed] (A) -- (E);
\draw (B) -- (F);
\draw (C) -- (G);
\draw (D) -- (H);

\draw[thick, blue, dotted] ($(C) ! 1/2 ! (G)$) -- ($(G) ! 1/3  !(F)$);
\draw[thick, red, dotted]  ($(G) ! 1/3  !(F)$) -- (E);

\draw[thick, blue, dotted] ($(C) ! 1/2 ! (G)$) -- ($(C) ! 1/3  !(D)$);
\draw[thick, red, dotted]  ($(C) ! 1/3  !(D)$) -- (A);

    \node at (I) {\tiny $\bullet$};

\end{tikzpicture}
\end{minipage}
\hspace{-0.7cm}
\begin{minipage}{0.33\linewidth}
\centering
\begin{tikzpicture}[scale=1.2]

 \coordinate (v) at (0.3,0.7);
 \coordinate (w) at (0.15,0.33);
     \coordinate (A) at (0,0);
     \coordinate (B) at (0.707,0.707);
     \coordinate (C) at (0,1.414);
     \coordinate (D) at (-0.707,0.707);     
     \coordinate (E) at ($(A) + (v)$);
          \coordinate (F) at ($(B) + (v)$);
               \coordinate (G) at ($(C) + (v)$);
                    \coordinate (H) at ($(D) + (v)$);
                    \coordinate (I) at ($(C) + (w)$);
                        \coordinate (J) at ($(A) + (w)$);
                        \coordinate (K) at ($(A) + (w) + (0,0.707)$);

\draw (A) -- (B) -- (C) -- (D) -- (A);
\draw[dashed]  (H) -- (E) --(F);
\draw (F) -- (G) -- (H);
\draw[dashed] (A) -- (E);
\draw (B) -- (F);
\draw (C) -- (G);
\draw (D) -- (H);

\draw[fill, opacity=0.2] (A) -- (E) --(G) -- (C) -- (A);

    \node at (I) {\tiny $\bullet$};
      \node at (J) {\tiny $\bullet$};
          \node at (K) {\tiny $\bullet$};
        \draw [very thick, dashed] (I) -- (J);

\end{tikzpicture}
\end{minipage}
\caption{Rotational symmetry is preserved by the refolding process. The axis of symmetry $L$ is dashed and the plane containing $e$ and $L$ is opaque.}\label{fig2}
\end{figure}

In the other two cases the point $m$ is the midpoint of an edge $e$. We again consider the embedding of the unfolding in $\R^3$ contained in the plane $z=0$ and  interpret the central symmetry at the rotation by $180^{\circ}$ around the line $L$ 
passing through $m$  which is  perpendicular to the plane $z=0$. Repeating the same procedure as above yields a trajectory which has been refolded everywhere
except along the edge $e$ (Figure \ref{fig2} bottom left). Consider the plane $P$ containing $e$ and the line $L$ and
 fold the edge $e$ in such a way that this plane is fixed, and  the plane becomes the bisector of the angle which is $90^{\circ}$.  The plane $P$ contains the opposite edge $e'$ to $e$ and the center of the cube, in fact  the line $L$ passes through the center and the midpoint of $e'$ (Figure \ref{fig2} right). Again the two endpoints of $\gamma$ are a rotation of each other and thus cannot coincide (Figure \ref{fig2} middle).

Now we adapt this argument to the other three polyhedra, all three of which are made of equilateral triangles, so the following argument applies to each.
We begin with a geodesic segment $\gamma$ which starts and ends at a vertex and unfold it to a straight line segment
starting and ending at vertices  of the equilateral triangle tiling of the plane whose  sides are parallel to the unit vectors
$v_1 := (1,0)$, $v_2 := (\frac12,\frac{\sqrt{3}}{2})$ and $v_3 := (\frac{-1}{2},\frac{\sqrt{3}}{2})$.
We  suppose that the unfolded trajectory goes from the origin to a point $(p,q)$.  Just as in the case of the cube,  the midpoint $m$ has coordinates $(p/2,q/2)$ in the basis $\{v_1,v_2\}$. The segment $\gamma$ is a geodesic, thus  either $p$ or $q$ must be odd, i.e., $m$ is of the form
$(k,l+\frac{1}{2})$, $(k+\frac{1}{2},l)$, or $(k+\frac{1}{2},l+\frac{1}{2})$.  In the first case $m$ is in the middle of an edge in the direction $v_2$, in the second case $m$ is in the middle of an edge in the direction $v_1$, while in the last case $m$ is in the middle of an edge in the direction $v_3$, in each of these cases the unfolding is centrally symmetric around the point $m$ (Figure \ref{fig3}).

\usetikzlibrary{lindenmayersystems}
\begin{figure}[ht]
\centering
\begin{tikzpicture}
  \pgfdeclarelindenmayersystem{triangular grid}{\rule{F->F-F+++F--F}}
  \path[draw=black,
  l-system={triangular grid,step=1cm,
    angle=-60,axiom=F--F--F,order=2,
  }]
  lindenmayer system -- cycle;
  
  \draw[dotted] (0,0) -- (7/2,0.866);
 \draw[densely dotted] (0,0) -- (5/2,0.866);
 \draw[dashed] (0,0) -- (2,1.732);
 \node at (7/4,0.433) {\tiny $\bullet$};
  \node at (5/4,0.433) {\tiny $\bullet$};
  \node at (1,0.866) {\tiny $\bullet$};

\end{tikzpicture}
\caption{The three possible cases for the triangular lattice.}\label{fig3}
\end{figure}
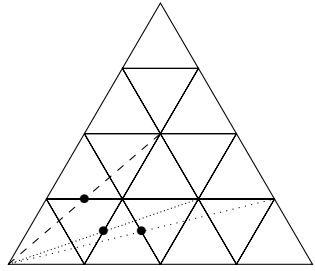

As in the previous case, embedding in the plane $z=0$ of $\R^3$, and refolding leads to an axis $L$ of rotational  symmetry of the geodesic $\gamma$. In each of the three cases the line of symmetry connects the midpoint of an edge to the center of the polyhedron and then to the midpoint of another edge, but in somewhat different ways.  The cases of the  octagon (Figure \ref{fig4} Left) and the  icosahedron are similar to the cube, if $m$ is the midpoint of the edge $e$, then the plane 
containing $L$ and $e$ contains the center $c$ of the polygon and another edge $e'$, and  $L$ passes through $m$, the midpoint $m'$ of $e'$ and $c$.
 The rotation by $180^\circ$ about $L$
 does not fix any vertices, thus the endpoints of $\gamma$ are distinct.
The tetrahedral case is different.  Again we consider the plane $P$ containing the edge $e$  and the line $L$, which after the final refold bisects the angle.  This bisection property implies that $P$  and in fact $L$ contain the center of the tetrahedron and the midpoint of the edge opposite to $e$ (Figure \ref{fig4} right).
 \end{proof}

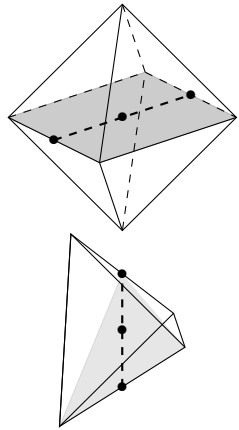
\begin{figure}
\begin{minipage}{0.5\linewidth}
 \centering
 \begin{tikzpicture}[scale=3]
\coordinate (A1) at (0,0);
\coordinate (A2) at (0.6,0.2);
\coordinate (A3) at (1,0);
\coordinate (A4) at (0.4,-0.2);
\coordinate (B1) at (0.5,0.5);
\coordinate (B2) at (0.5,-0.5);

\draw[dashed] (A1) -- (A2) -- (A3);
\draw (A1) -- (A4) -- (A3);
\draw[dashed] (B1) -- (A2) -- (B2);
\draw (B1) -- (A4) -- (B2);
\draw (B1) -- (A1) -- (B2) -- (A3) --cycle;

\node at (0.2,-0.1) {\tiny $\bullet$};
\node at (0.8,0.1) {\tiny $\bullet$};
\node at (0.5,0) {\tiny $\bullet$};
\draw[fill, opacity=0.2] (A1) -- (A2) -- (A3) -- (A4) -- cycle;
\draw[thick, dashed] (0.2,-0.1) -- (0.8,0.1)  ;

\end{tikzpicture}
\end{minipage}
\begin{minipage}{0.5\linewidth}
 \centering
\begin{tikzpicture}[scale=1.5]
\coordinate (A1) at (0,0);
\coordinate (A2) at (1,0);
\coordinate (A3) at (1,1);
\coordinate (A4) at (0,1);

\coordinate (w') at (0.05, 0.35);
\coordinate (v) at  ($2*(w')$);

\coordinate (w) at ($(0.5,0) + (w')$);

\coordinate (B1) at  ($(A1) + (v)$);
\coordinate (B2) at  ($(A2) + (v)$);
\coordinate (B3) at  ($(A3) + (v)$);
\coordinate (B4) at  ($(A4) + (v)$);

\coordinate (m) at ($(A1) + (w)$);
\coordinate (c) at ($(0.5,0.5) + (w')$);
\coordinate (n) at ($(c) + (c) - (m)$);

\node  at (m) {\tiny $\bullet$};
\node  at  (c) {\tiny $\bullet$};
\node at (n) {\tiny $\bullet$};

\draw (A1) -- (A3) -- (B4) -- cycle;
\draw  (A1) -- (B2)  -- (B4) -- cycle;
\draw (A3) -- (B2);
\draw[fill, opacity = 0.1] (n) -- (A1) -- (B2);

\draw[thick, dashed] (m) -- (c) -- (n);
\end{tikzpicture}
\end{minipage}

\caption{The plane $P$ and the axis $L$ pass through the midpoints of the two edges and the center of the polygon.}\label{fig4}
\end{figure}

There is one more platonic solid, the dodecahedron. It turns out that there are geodesics from a vertex to itself on the dodecahedron \cite{AA,AA1,F,P}, of course our proof cannot work in this case since pentagons do not tile the plane.

The symmetries of platonic solids have been extremely well studied.  In particular, any pair of vertices is symmetric by one of the symmetries arising in our proof, one can  construct explicit geodesics between any pair of distinct vertices, such constructions have been given in
 in \cite{DDTY} and \cite{F}.

On the tetrahedron there is a simpler proof since the net of the tetrahedron  is an equilateral triangle, so it tiles the plane.  Thus each vertex of the triangular lattice corresponds to a unique vertex of the polygon once we fix the correspondance at the origin.  This is not the case for the three other platonic solids we treat since their nets do not tile the plane .

\end{document}